\documentclass[a4paper]{amsart}
\usepackage[T1]{fontenc}
\usepackage[utf8]{inputenc}
\usepackage{amssymb, amscd, mathrsfs, xifthen, xfrac, url}
\usepackage[colorinlistoftodos]{todonotes}
\newtheorem{thm}{Theorem}

\newtheorem{lem}[thm]{Lemma}
\newtheorem{prop}[thm]{Proposition}
\newtheorem{cor}[thm]{Corollary}

\theoremstyle{definition}
\newtheorem*{df}{Definition}
\theoremstyle{remark}
\newtheorem*{rem}{Remark}
\newcommand{\gb}{\mathfrak{b}}
\newcommand{\gc}{\mathfrak{c}}
\newcommand{\go}{\mathfrak{o}}
\newcommand{\gp}{\mathfrak{p}}

\newcommand{\gq}{\mathfrak{q}}
\newcommand{\gr}{\mathfrak{r}}

\newcommand{\FF}{\mathbb{F}}
\newcommand{\ZZ}{\mathbb{Z}}
\newcommand{\CO}{\mathcal{O}}
\newcommand{\GE}{\mathscr{E}}
\newcommand{\GC}{\mathfrak{C}}
\newcommand{\GB}{\mathfrak{B}}
\newcommand{\SD}{\mathcal{D}}
\newcommand{\SB}{\mathscr{B}}
\newcommand{\card}[1]{|#1|}
\newcommand{\into}{\rightarrowtail}
\newcommand{\iso}{\xrightarrow{\raisebox{-2bp}{\smash{\tiny$\,\sim\,$}}}}
\newcommand{\units}[1]{#1^\times}
\newcommand{\squares}[1]{#1^{\times2}}
\newcommand{\Sing}[1]{\mathbf{E}_{#1}}
\newcommand{\SingX}{\Sing{X}}
\newcommand{\SingXp}{\Sing{X\setminus \{\gp\}}}
\newcommand{\SingY}{\Sing{Y}}

\newcommand{\Dl}[1]{\mathbf{\Delta}_{#1}} 
\newcommand{\st}{\mathrel{\mid}}
\DeclareMathOperator{\Div}{Div}
\newcommand{\DivX}{\Div X}
\newcommand{\DivY}{\Div Y}
\DeclareMathOperator{\Pic}{Pic}
\newcommand{\PicX}{\Pic X}
\newcommand{\PicZeroX}{\Pic^0 X}
\newcommand{\PicY}{\Pic Y}

\newcommand{\PPic}[1]{\sfrac{\Pic #1}{2\Pic #1}}
\newcommand{\PPicX}{\PPic{X}}
\newcommand{\PPicY}{\PPic{Y}}
\DeclareMathOperator{\rank}{rk}
\newcommand{\rk}[1][]{\ifthenelse{\equal{#1}{}}{\rank_{2}}{\rank_{#1}}}
\newcommand{\class}[2][]{\ifthenelse{\equal{#1}{}}{[#2]}{[#2]_#1}}   
\DeclareMathOperator{\dv}{div}
\newcommand{\divX}{\dv_X}
\newcommand{\divY}{\dv_Y}
\DeclareMathOperator{\ord}{ord}
\newcommand{\un}[1][]{\ifthenelse{\equal{#1}{}}{\units{K}}{\units{#1}}}
\newcommand{\sqgd}[1]{\sfrac{\units{#1}}{\squares{#1}}}
\newcommand{\term}[1]{\emph{#1}}
\newcommand{\la}{\lambda}

\newcommand{\lap}{\la_\gp}
\newcommand{\laq}{\la_\gq}
\newcommand{\even}{\equiv 0\pmod{2}}
\newcommand{\odd}{\equiv 1\pmod{2}}
\DeclareMathOperator{\dcup}{\dot{\cup}}
\newcommand{\ideal}[1]{\langle#1\rangle}
\DeclareMathOperator{\Gal}{Gal}
\newcommand{\Artin}[3]{\genfrac{(}{)}{}{}{#1/#2}{#3}}  
\newcommand{\Legendre}[2]{\genfrac{(}{)}{}{}{#1}{#2}}
\def\clap#1{\hbox to 0pt{\hss#1\hss}}

\def\mathclap{\mathpalette\mathclapinternal}

\def\mathclapinternal#1#2{\clap{$\mathsurround=0pt#1{#2}$}}
\newwrite\refs
\openout\refs=\jobname.refs
\makeatletter
\renewcommand\@setref[3]{%
        \ifx#1\relax
                \write\refs{'#3' \thepage\space undefined}%
                \protect \G@refundefinedtrue
                \nfss@text{\reset@font\bfseries ??}%
                \@latex@warning{Reference `#3' on page \thepage\space undefined}%
        \else
                \write\refs{'#3' \thepage\space \expandafter\@secondoftwo#1}%
                \expandafter#2#1\null
        \fi
}
\makeatother
\title{Graph of even points on an arithmetic curve}
\author[A. Czoga{\l}a \and P. Koprowski]{Alfred Czoga\l a \and Przemys{\l}aw Koprowski}
\address{Institute of Mathematics\\
University of Silesia\\ Bankowa 14\\ 40-007 Katowice,
Poland} \email{alfred.czogala@us.edu.pl}
\address{Institute of Mathematics\\
University of Silesia\\ Bankowa 14\\ 40-007 Katowice,
Poland} \email{przemyslaw.koprowski@us.edu.pl}
\begin{document}

\begin{abstract}
We show that the points of a global function field, whose classes are $2$-divisible in the Picard group, form a connected graph, with the incidence relation generalizing the well known quadratic reciprocity law. We prove that for every global function field the dimension of this graph is precisely~$2$. In addition we develop an analog of global square theorem that concerns points $2$-divisible in the Picard group.
\end{abstract}
\maketitle

\section{Introduction}
Let~$\FF_q$ be a finite field of odd characteristic. Further let $f,g\in \FF_q[t]$ be two irreducible polynomials. The well known quadratic reciprocity law says that
\[
\Legendre{f}{g}\Legendre{g}{f} = \bigl(-1\bigr)^{\frac{|f|-1}{2}\cdot \frac{|g|-1}{2}},
\]
where $|f| = q^{\deg f}$ (respectively $|g| = q^{\deg g}$) is the cardinality of the residue field $\sfrac{\FF_q[t]}{\ideal{f}}$ (resp. $\sfrac{\FF_q[t]}{\ideal{g}}$). Thus we may define a relation on the set of irreducible polynomials. We say that a polynomial~$f$ is related to~$g$, written $f\smile g$, when $\Legendre{f}{g} = 1$. This relation is symmetric unless $-1$ is a quadratic non-residue simultaneously modulo~$f$ and~$g$. In the later case $\smile$ is antisymmetric.

If we consider ideals instead of polynomials, this relation remains to be well defined, providing that the generators have even degrees. Indeed, take two ideals: $\gp = \ideal{f}$ and $\gq = \ideal{g}$, where $\deg f, \deg g\in 2\ZZ$. Write $\gp\smile \gq$ if the relation holds for the generators, i.e. if $f\smile g$. If $uf$ is another generator of~$\gp$, then $u\in \units{\FF_q}$ and so $\Legendre{u}{g} = 1$, since the degree of~$g$ is even. This shows that $\smile$ is well defined on the set of prime ideals of even degrees.
The quadratic reciprocity law ensures that this relation is symmetric. It is not, however, transitive as the following example shows. Take three polynomials $f,g,h$ with coefficients in $\FF_5$:
\[
f = t^{2} + 4 t + 1,\qquad g = t^{2} + 2 t + 3, \qquad h = t^{2} + 2.
\]
One easily checks that $f\smile g$ and $g\smile h$, yet still $f\not\smile h$. Nevertheless it can be shown that for every two unrelated polynomials $f, g$, there exists a third polynomial~$h$ related to both of them simultaneously. In particular, if we consider a graph, whose vertices are prime ideals of even degrees and $smile$ is the incidence relation, then this graph turns out to be connected and its diameter equals~$2$. The main aim of this paper it to prove that this property holds not just for polynomials and rational functions but in \emph{every} global function field (Theorem~\ref{thm:diamG=2}).

Throughout this paper $\FF_q$ is a finite field of odd characteristic and $K = \FF_q(X)$ is an algebraic function field of one variable over~$\FF_q$ (i.e. a global function field). The set~$X$ of classes of valuation of~$K$ can be treated as a smooth complete curve. For a (closed) point $\gp\in X$, by $\class{\gp}$ we denote its class in the Picard group of~$X$. We say that $\gp\in X$ is an \term{even point}, if its class in the Picard group is $2$-divisible, i.e. if $\class{\gp}\in 2\PicX$. In our introductory example, when $K = \FF_q(t)$ is the field of rational functions, the even points are precisely the ideals generated by polynomials of even degrees since $\deg : \PicX\to \ZZ$ is a group isomorphism then.

Even points exist in every global function field (Theorem~\ref{thm:density_of_EP}). This notion emerges naturally when one studies theory of quadratic forms over global function fields. It was proved in \cite{CKR18} that a point is even if and only if it is a unique wild point of some self-equivalence of~$K$. 

The set of even points admit a symmetric relation~$\smile$ (see Section~\ref{sec:graph} for a definition) generalizing the one defined above for polynomials. Like in the case of polynomials this relation is symmetric (see Proposition~\ref{prop:symmetry}) but not transitive (see \cite[Remark~3]{CKR18}). It is known (see \cite[Proposition~4.5]{CKR18} and \cite[Proposition~4.7]{CKR19}) that the relation~$\smile$ controls the formation of bigger wild sets of self-equivalences of~$K$. 

The main results of this paper are:
\begin{itemize}
\item Theorem~\ref{thm:diamG=2} saying that the graph of even points of a global function field is always connected and its diameter is precisely~$2$.
\item Theorem~\ref{thm:GST} which establishes an even point analog of the Global Square Theorem.
\item Theorem~\ref{thm:density_of_EP} showing that the set of even points have positive density, hence the even points exist in every global function field.
\item Theorem~\ref{thm:EE_X=PPicX} which says that the quotient group $\PPicX$ is naturally isomorphic to the group of square classes of~$K$ that have even valuations everywhere on~$X$.
\end{itemize}
The problem of divisibility in the Picard groups (including Picard groups of arithmetic curves) has been quite vivid in recent years and has drawn attention of numerous authors. We may cite for example \cite{Farkas2010, GJRW1996, Putman2012, Sharif2013}. The list is definitely very far from being complete but it already gives the reader some glimpse of the subject.

\section{Notation}
In what follows we use the following notation (partially already introduced above). Let $\gp\in X$ be a (closed) point, by $\ord_\gp$ we denote the associated valuation. Further $\CO_\gp = \{\la\in K\st \ord_\gp\la\geq 0\}$ is the valuation ring, $K_\gp$ is the completion of~$K$ at~$\gp$ and $K(\gp)$ is the residue field. If $Y\subseteq X$ is an open subset of~$X$ and $\SD\in \DivY$ is a divisor, then $\class[Y]{\SD}$ denotes its class in the Picard group $\PicY$. In case $Y = X$ we drop the subscript and simply write $\class{\SD}$.

It is well known that the exact sequence
\[
0\to \PicZeroX\to \PicX\xrightarrow{\deg} \ZZ\to 0
\]
splits (because~$\ZZ$ is a projective $\ZZ$-module). Equivalently we may write
\begin{equation}\label{eq:PicX_decomposition}
\PicX\cong \PicZeroX\oplus \ZZ,
\end{equation}
where the projection onto the second coordinate is the degree homomorphism. Therefore the necessary condition for a point $\gp\in X$ to be even is $\deg\gp\in 2\ZZ$. This condition is not sufficient, though, unless $\card{\PicZeroX}$ is odd. A convenient condition of evenness of~$\gp$ makes use of a certain subgroup of the square class group of~$K$. Let $Y\subseteq X$ be an open nonempty subset of~$X$. The group morphism $\divY: \un\to \DivY$, that assigns to a nonzero element of~$K$ its principal divisor, induces a morphism of the quotient groups $\sqgd{K}\to \sfrac{\DivY}{2\DivY}$. Harmlessly abusing the notation, we denote the latter morphism by $\divY$, too. Define the subgroup~$\SingY$ of~$\sqgd{K}$ to be the kernel of this map:
\[
\SingY := \ker\bigl( \divY: \sqgd{K}\to \sfrac{\DivY}{2\DivY} \bigr).
\]
It is clear that~$\SingY$ consists of these square classes that have even valuations everywhere on~$Y$:
\[
\SingY = \bigl\{ \la\in\sqgd{K}\st \ord_\gp\la\even \quad\text{for all }\gp\in Y\bigr\}.
\]
If $Y$ is a proper subset of~$X$, we define one more subgroup. Let~$\Dl{Y}$ be the subset of~$\SingX$ consisting of classes of functions that are local squares outside~$Y$, namely
\[
\Dl{Y} := \SingX \cap \bigcap_{\gp\notin Y}\squares{K_\gp} = \SingY \cap \bigcap_{\gp\notin Y}\squares{K_\gp}.
\]
The groups $\Dl{Y}, \SingY, \sqgd{K}$ and $\PPicX$ are elementary $2$-groups. We will often treat them as $\FF_2$-vector spaces.

\section{Natural isomorphism}
We begin with a function field analog of a result, which for number fields is known as Hecke Satz 169. We suspect that it may be known to experts but we are not aware of any convenient reference.

\begin{thm}[``Hecke Satz 169'' for function fields]\label{thm:Hecke_Satz_169}
Let $K = \FF_q(X)$ be a global function field. Further let $\la_1, \dotsc, \la_n\in \sqgd{K}$ be square classes linearly independent over~$\FF_2$ and $e_1, \dotsc, e_n\in \{0,1\}$ be some arbitrary exponents. Then the density of a set
\[
A := \Bigl\{ \gp\in X\st \Legendre{\la_i}{\gp} = (-1)^{e_i}\quad\text{for every }1\leq i\leq n\Bigr\}
\]
equals $\delta A = \frac{1}{2^n}$.
\end{thm}

\begin{proof}
For every $i\leq n$ let $L_i := K(\sqrt{\la_i})$ be the quadratic extension of~$K$ determined by~$\la_i$. Further let $G_i := \Gal(L_i/K)$ be the associated Galois group and $\sigma_i\in G_i$ be the unique non-trivial $K$-automorphism of~$L_i$. Take the composite field $L := L_1\dotsm L_n = K\bigl(\sqrt{\la_1}, \dotsc, \sqrt{\la_n}\bigr)$. It is a multi-quadratic extension of~$K$ and it follows from Kummer theory that $L/K$ is abelian with the Galois group
\begin{equation}\label{eq:multiquad_Galois}
G := \Gal(L/K) \cong G_1\times \dotsb \times G_n\cong (\sfrac{\ZZ}{2\ZZ})^n.
\end{equation}
By Chebotarev density theorem (see e.g. \cite[Theorem~9.13A]{Rosen02} or \cite[Theorem~12, Chapter~XII]{Weil95}) for every $\tau\in G$ we have
\begin{equation}\label{eq:density}
\delta\Bigl\{ \gp\in X\st \Artin{L}{K}{\gp} = \tau\Bigr\} = \frac{1}{\card{G}} = \frac{1}{2^n}.
\end{equation}
Here $\Artin{L}{K}{\gp}$ is the Artin symbol of~$\gp$ in the field extension $L/K$. Take an automorphism $\sigma := \bigl( \sigma_1^{e_1}, \dotsc, \sigma_n^{e_n}\bigr)\in G$. Using the isomorphism Eq.~\eqref{eq:multiquad_Galois} we treat the Artin symbol $\Artin{L}{K}{\gp}$ as a tuple $( \Artin{L_1}{K}{\gp}, \dotsc, \Artin{L_n}{K}{\gp} )$. Fix an index $i\leq n$. We claim that $\Artin{L_i}{K}{\gp} = \sigma_i^{e_i}$ if and only if $\Legendre{\la_i}{\gp} = (-1)^{e_i}$. Indeed, the group $G_i$ consists of just two elements: the identity and $\sigma_i$. The Artin symbol $\Artin{L_i}{K}{\gp}$ vanishes if and only if $\gp$ splits in $L_i$ if and only if the polynomial $t^2- \la_i(\gp)$ factors over the residue field $K(\gp)$. This last condition is equivalent to $\Legendre{\la_i}{\gp} = 1$, which proves the claim. It follows now from Eq.~\eqref{eq:density} that
\[
\frac{1}{2^n} 
= \delta\Bigl\{ \gp\in X\st \Artin{L}{K}{\gp} = \sigma \Bigr\} 
= \delta\Bigl\{ \gp\in X\st \Legendre{\la_i}{\gp} = (-1)^{e_i}\quad\text{for }i\leq n\Bigr\}. \qedhere
\]
\end{proof}

\begin{df}
Let $\gp_1, \dotsc, \gp_k\in X$ and $\la_1, \dotsc, \la_k\in \SingX$. We say that the points $\gp_1, \dotsc, \gp_k$ are \term{compatible} with square classes $\la_1, \dotsc, \la_k$ if 
\[
\Legendre{\la_i}{\gp_j} =
\begin{cases}
1 &\text{if }i\neq j\\
-1 &\text{if }i = j
\end{cases}
\]
for all $1\leq i, j\leq n$.
\end{df}

\begin{lem}\label{lem:points_to_classes}
Let $\gp_1, \dotsc, \gp_k\in X$ be finitely many points. If the classes $\class{\gp_1}, \dotsc, \class{\gp_k}$ are linearly independent \textup(over~$\FF_2$\textup) in $\PPicX$, then there exist linearly independent elements $\la_1, \dotsc, \la_k\in \SingX$ compatible with $\gp_1, \dotsc, \gp_k$.
\end{lem}

\begin{proof}
We must consider two cases. First assume that all the points $\gp_1, \dotsc, \gp_k$ have even degrees. A classical theorem by F.K.~Schmidt (see e.g. \cite[Corollary~V.1.11]{Sti93}) implies  that there is a point $\go\in X$ of an odd degree. Take an affine curve $Y := X\setminus \{\go\}$. By the assumption, $\class{\gp_1}, \dotsc, \class{\gp_k}$ are linearly independent in $\PPicX$. We claim that also $\class[Y]{\gp_1}, \dotsc, \class[Y]{\gp_k}$ are linearly independent in $\PPicY$.

In order to prove the claim suppose that
\[
\varepsilon_1\class[Y]{\gp_1} + \dotsb + \varepsilon_k\class[Y]{\gp_k} \equiv 0\pmod{2\PicY},
\]
for some $\varepsilon_1, \dotsc, \varepsilon_k\in \FF_2$. Thus there exists a divisor $\SD\in \DivY$ and an element $\mu\in \un$ such that the equality
\[
\divY\mu = \sum_{i\leq k}\varepsilon_i\gp_i + 2\SD
\]
holds in the group $\DivY$. Passing to the divisor group of the complete curve we write
\[
\divX\mu = \sum_{i\leq k}\varepsilon_i\gp_i + 2\SD + \ord_\go(\mu)\cdot \go.
\]
Compute the degrees of both sides to get
\[
0 = \sum_{i\leq k}\varepsilon_i\cdot \deg\gp_i + 2\deg \SD + \ord_\go(\mu) \cdot \deg \go.
\]
The degrees $\deg\gp_1, \dotsc, \deg\gp_k$ are all even, while $\deg\go$ is odd. It follows that $\ord_\go\mu = 2k$ for some $k\in\ZZ$. Consequently we obtain 
\[
\divX\mu = \sum_{i=1}^k\varepsilon_i\gp_i + 2(\SD + k\go). 
\]
Hence $\varepsilon_1\class{\gp_1} + \dotsb + \varepsilon_k\class{\gp_k}$ is the neutral element of the vector space $\PPicX$, therefore $\varepsilon_1 = \dotsb = \varepsilon_k = 0$, which proves the claim.

Now, \cite[Lemma~4.1]{CKR18} asserts that there exist linearly independent square classes $\la_1, \dotsc, \la_k\in \Dl{Y}$ compatible with $\gp_1, \dotsc, \gp_k$. The elements $\la_1, \dotsc, \la_k$ remain linearly independent in~$\SingX$ since $\Dl{Y}$ is a subspace of $\SingX$. This ends the proof in the first case.

We now consider the case when at least one of the points $\gp_1, \dotsc, \gp_k$ has an odd degree. Without loss of generality we may assume that it is the point~$\gp_k$. Set $Y := X\setminus\{\gp_k\}$. The classes $\class[Y]{\gp_1}, \dotsc, \class[Y]{\gp_{k-1}}$ are linearly independent in $\PPicY$ and so using \cite[Lemma~4.1]{CKR18} again we obtain square classes $\la_1, \dotsc, \la_{k-1}\in \Dl{Y}$ compatible with $\gp_1, \dotsc, \gp_{k-1}$. By the definition of $\Dl{Y}$ we see that $\Legendre{\la_i}{\gp_k}=1$ for every $i < k$. It remains to show the existence of~$\la_k$. Let~$\zeta$ be the unique non-trivial element of the square-class group $\sqgd{\FF_q}$. Observe that for every point $\gq\in X$ we have 
\[
\Legendre{\zeta}{\gq} =
\begin{cases}
1&\text{if }\deg\gq\even\\
-1&\text{if }\deg\gq\odd.
\end{cases}
\]
Define~$\la_k$ to be a product of~$\zeta$ multiplied by these $\la_i$'s that correspond to points of odd degrees:
\[
\la_k := \zeta\cdot \prod_{\mathclap{\substack{1\leq i< k\\ \deg\gp_i\notin 2\ZZ}}} \la_i.
\]
It is now clear that $\la_1, \dotsc, \la_{k-1}, \la_k$ are compatible with $\gp_1, \dotsc, \gp_k$.
\end{proof}

We now prove a converse of Lemma~\ref{lem:points_to_classes}.

\begin{lem}\label{lem:classes_to_points}
Let $\gp_1, \dotsc, \gp_k\in X$ be finitely many points. If there exists compatible square classes $\la_1, \dotsc, \la_k\in \SingX$, then the classes of $\gp_1, \dotsc, \gp_k$ are linearly independent in $\PPicX$.
\end{lem}

\begin{proof}
Suppose that $\varepsilon_1\class{\gp_1} + \dotsb + \varepsilon_k\class{\gp_k}\in 2\PicX$ for some $\varepsilon_1, \dotsc, \varepsilon_k\in \FF_2$ not all zero. Thus there exists a divisor $\SD\in \DivX$ and an element $\mu\in \un$ such that
\[
\divX\mu = \varepsilon_1\gp_1 + \dotsb + \varepsilon_k\gp_k + 2\SD.
\]
Set $j := \min \{i\st \varepsilon_i\neq 0\}$. In particular we have $\ord_{\gp_j}\mu\odd$. Observe that the valuation of~$\mu$ at every $\gq\in X\setminus \{\gp_1, \dotsc, \gp_k\}$ is necessarily even. On the other hand the valuation of~$\la_j$ is even absolutely everywhere as $\la_j\in \SingX$. Combining these two facts we see that the Hilbert symbol $(\mu, \la_j)_\gq$ vanishes everywhere except possibly at $\gp_1, \dotsc, \gp_k$.  Moreover, the compatibility between $\gp_1, \dotsc, \gp_k$ and $\la_1, \dotsc, \la_k$ implies that $\la_j$ is a local square at~$\gp_i$  for every $i\neq j$ and so $(\mu, \la_j)_{\gp_i} = 1$. Hilbert reciprocity law implies that:
\[
1 = \prod_{\gq\in X} (\mu, \la_j)_\gq = (\mu, \la_j)_{\gp_j}.
\]
But this is impossible since $\ord_{\gp_j}\mu\odd$ and $\la_j\notin \squares{K_{\gp_j}}$. 
\end{proof}

We are now ready to present the first of the main results of this paper.

\begin{thm}\label{thm:EE_X=PPicX}
$\SingX\cong \PPicX$.
\end{thm}

\begin{proof}
It is well known that the group $\PicZeroX$ is finite. Therefore the group~$\SingX$, viewed as a $\FF_2$-vector space, is finitely dimensional by \cite[Lemma~2.4]{CKR18}. Pick a basis $\la_1, \dotsc, \la_k\in \SingX$ of this vector space. Theorem~\ref{thm:Hecke_Satz_169} implies that there are compatible points $\gp_1, \dotsc, \gp_k\in X$ so we have
\[
\Legendre{\la_i}{\gp_j} = 
\begin{cases}
1&\text{if }i\neq j\\
-1&\text{i }i = j
\end{cases}
\]
for all $1\leq i,j\leq k$. Now, Lemma~\ref{lem:classes_to_points} asserts that the classes $\class{\gp_1}, \dotsc, \class{\gp_k}$ are linearly independent in $\PPicX$. Take a linear map that sends $\la_i$ to~$\gp_i$ and extend it by linearity. It is an injection $\SingX\into \PPicX$ of $\FF_2$-vector spaces. We claim that it is actually an isomorphism. Indeed, suppose a contrario that this function is not surjective. Thus, there is a point $\gp_{k+1}\in X$ such that $\class{\gp_1}, \dotsc, \class{\gp_k}, \class{\gp_{k+1}}$ remain linearly independent.  Lemma~\ref{lem:points_to_classes} says that there are $k+1$ linearly independent elements in~$\SingX$, but this contradicts the fact that we have $\dim_{\FF_2}\SingX = k$.
\end{proof}

\begin{prop}\label{prop:Legendre_from_coords}
Let $\GB = \{\gb_1, \dotsc, \gb_k\}$ be a subset of~$X$ such that $\class{\GB} := \bigl\{ \class{\gb_1}+2\PicX, \dotsc, \class{\gb_k}+2\PicX\bigr\}$ is a basis of $\PPicX$. Further let $\SB = \{\beta_1, \dotsc, \beta_k\}$ be a compatible basis of~$\SingX$. For every $\gp\in X$ and $\la \in \SingX$ if $\varepsilon_1, \dotsc, \varepsilon_k$ are the coordinates of $\class{\gp} + 2\PicX$ with respect to~$\class{\GB}$ and $\epsilon_1, \dotsc, \epsilon_k$ are the coordinates of~$\la$ with respect to $\SB$, then
\[
\Legendre{\la}{\gp} = (-1)^{\sum \varepsilon_i\epsilon_i}.
\]
\end{prop}

\begin{proof}
The assertion is trivially true when $\gp\in \GB$, thus without loss of generality we may assume that $\gp\in X\setminus \GB$. By the assumption we have
\[
\class{\gp} \equiv \sum_{i\leq k}\varepsilon_i\class{\gb_i} \pmod{2\PicX}
\qquad\text{and}\qquad
\la = \prod_{j\leq k}\beta_j^{\epsilon_j}.
\]
Therefore there is a divisor $\SD\in \DivX$ and an element $\mu\in\un$ such that
\[
\divX\mu = \gp + \sum_{i\leq k}\varepsilon_i\gb_i + 2\SD.
\]
Using Hilbert reciprocity law, we write
\[
(\la,\mu)_\gp
= \prod_{\gq\neq \gp} (\la,\mu)_\gq
= \prod_{i\leq k} (\la,\mu)_{\gb_i}^{\varepsilon_i}
= \prod_{i,j\leq k} (\beta_j,\mu)_{\gb_i}^{\varepsilon_i\epsilon_i}
\]
The element~$\beta_j$ is a local square at every~$\gb_i$ except~$\gb_j$. Hence the above formula simplifies to:
\[
(\la,\mu)_\gp
= \prod_{i\leq k}(-1)^{\varepsilon_i\epsilon_i} = (-1)^{\sum\varepsilon_i\epsilon_i}.
\]
Now $\ord_\gp\mu\odd$ and $\la\in \SingX$, hence $\ord_\gp\la\even$. Therefore the Hilbert symbol $(\la,\mu)_\gp$ is the same as the Legendre symbol $\Legendre{\la}{\gp}$.
\end{proof}

\begin{prop}
If $\gp, \gq\in X$ are two points whose classes are congruent modulo $2\PicX$. Then $\Dl{X\setminus\{\gp\}} = \Dl{X\setminus\{\gq\}}$.
\end{prop}

\begin{proof}
Fix a subset $\GB = \{\gb_1, \dotsc, \gb_k\}$ of~$X$ such that $\class\GB$ is a basis of $\PPicX$. Further let $\SB = \{\beta_1, \dotsc, \beta_k\}\subset \SingX$ be a compatible basis of~$\SingX$. By the assumption we have
\[
\class{\gp} + 2\PicX = \class{\gq} + 2\PicX.
\]
Hence there are $\varepsilon_1, \dotsc, \varepsilon_k\in \FF_2$ such that
\[
\class{\gp} \equiv \class{\gq} \equiv \sum_{i\leq k}\varepsilon_i\class{\gb_i}\pmod{2\PicX}.
\]
Take any $\mu\in \Dl{X\setminus\{\gp\}}\subset \SingX$ and let $\epsilon_1, \dotsc, \epsilon_k$ be its coordinates with respect to~$\SB$. We then have
\[
1 = \Legendre{\mu}{\gp} = (-1)^{\sum\varepsilon_i\epsilon_i} = \Legendre{\mu}{\gq}.
\]
This means that $\mu\in \squares{K_\gq}\cap \SingX = \Dl{X\setminus\{\gq\}}$ and so $\Dl{X\setminus\{\gp\}}\subseteq \Dl{X\setminus\{\gq\}}$. The opposite inclusion follows by symmetry.
\end{proof}

\begin{prop}\label{prop:Pic_class_by_Legendre}
Let~$\GB$ and~$\SB$ be as in Proposition~\ref{prop:Legendre_from_coords}. The class of a point $\gp\in X$ has coordinates $\varepsilon_1, \dotsc, \varepsilon_k$ with respect to the basis~$\class\GB$ if and only if
\[
\Legendre{\beta_i}{\gp} = (-1)^{\varepsilon_i}\qquad\text{for every }1\leq i \leq k.
\]
\end{prop}

\begin{proof}
Let $\class{\gp} \equiv \sum_{i\leq k}\varepsilon_i \class{\beta_i}\pmod{2\PicX}$, it follows from Proposition~\ref{prop:Legendre_from_coords} that
\[
\Legendre{\beta_i}{\gp} = (-1)^{\varepsilon_1\cdot 0 + \dotsb + \varepsilon_i\cdot 1 + \dotsb + \varepsilon_k\cdot 0} = (-1)^{\varepsilon_i}.
\]
This proves one implication. Conversely, assume that $\Legendre{\beta_i}{\gp} = (-1)^{\varepsilon_i}$ for $i\leq k$ and the coordinates of the class of~$\gp$ are $\varepsilon_1', \dotsc, \varepsilon_k'\in \FF_2$. By the previous part, for every $i\leq k$ we have
\[
(-1)^{\varepsilon_i} = \Legendre{\beta_i}{\gp} = (-1)^{\varepsilon_i'}.
\]
Consequently $\varepsilon_i = \varepsilon_i'$.
\end{proof}

\begin{cor}\label{cor:uniq_compatible_basis}
For every basis of~$\SingX$ a compatible basis of $\PPicX$ is determined uniquely.
\end{cor}

\begin{proof}
Fix a basis $\SB = \{\beta_1, \dotsc, \beta_k\}$ of~$\SingX$. Let $\class\GB = \bigl\{\class{\gb_i}+2\PicX\st i\leq k\bigr\}$ and $\class\GC = \bigl\{\class{\gc_i}+ 2\PicX\st i\leq k\bigr\}$ be two bases of $\PPicX$ determined by some subsets
\[
\GB = \{\gb_1, \dotsc, \gb_k\},\qquad \GC = \{\gc_1, \dotsc, \gc_k\}
\]
of~$X$. By the compatibility we have
\[
\Legendre{\beta_j}{\gb_i} = \Legendre{\beta_j}{\gc_i}
\]
for all $i,j\leq k$. Therefore Proposition~\ref{prop:Pic_class_by_Legendre} implies that the points $\gb_i$ and~$\gc_i$ are congruent modulo~$2\PicX$ for every index~$i$.
\end{proof}

\begin{rem}
Let $\GB$ and~$\SB$ be as in Proposition~\ref{prop:Legendre_from_coords}. In the proof of Theorem~\ref{thm:EE_X=PPicX} we constructed an isomorphism $\Phi:\SingX\iso \PPicX$ that sent $\beta_i\in \SB$ to $\class{\gb_i} + 2\PicX$. Corollary~\ref{cor:uniq_compatible_basis} asserts that the basis~$\class\GB$ of $\PPicX$ --- hence also the isomorphism~$\Phi$ --- are uniquely determined by~$\SB$. One may wonder whether $\Phi$ really depends on the choice of the basis~$\SB$. In fact it does. Different bases give rise to different $\Phi$'s. To see this phenomenon happen, take a basis $\SB = \{\beta_1, \dotsc, \beta_k\}$ of~$\SingX$ and the (unique) basis $\class{\GB}$ of $\PPicX$, where $\GB = \{ \gb_1, \dotsc, \gb_k\}\subset X$. Take the isomorphism $\Phi:\SingX\iso \PPicX$ that sends $\beta_i$ to $\class{\gb_i} + 2\PicX$. Now, construct another basis $\SB' := \{\beta_1', \dotsc, \beta_k'\}$ of~$\SingX$ setting
\[
\beta_1' := \beta_1\beta_2\qquad\text{and}\qquad \beta_i' := \beta_i\quad\text{for }i\geq 2.
\]
Further let $\gp\in X$ be such that
\[
\class{\gp}\equiv \Phi(\beta_1')\equiv \class{\gb_1 + \gb_2}\pmod{2\PicX}.
\]
Using Proposition~\ref{prop:Legendre_from_coords} we compute
\[
\Legendre{\beta_1'}{\gp} = (-1)^{1\cdot 1 + 1\cdot 1} = 1.
\]
Therefore the class of~$\gp$ cannot be the first vector of a basis of~$\PPicX$ compatible with~$\SB'$. This proves that the isomorphism $\Phi'$ associated to the basis~$\SB'$ differ from~$\Phi$.
\end{rem}

\section{Global square theorem}
The well known Global Square Theorem (GST) says that an element of a global field is locally a square \emph{almost everywhere} if and only if it is a square globally if and only if it is a local square \emph{everywhere}. In this section we derive an analog of this theorem for even points. First, however, we gather a few criteria for evenness of a point. Some of them were already proved in \cite{CKR18}. Nevertheless we repeat them here to make the paper self-contained and easier to read.

\begin{prop}\label{prop:even_criteria}
Let $\gp\in X$ be a point. Denote $Y := X\setminus \{\gp\}$. The following conditions are equivalent:
\begin{enumerate}
\item\label{it:even} $\gp$ is even,
\item\label{it:exists_la} there is $\lap\in \SingXp$ such that $\ord_\gp\lap\odd$,
\item\label{it:EE_X=Dl_Y} $\SingX = \Dl{Y}$,
\item\label{it:EE_X=PPicY} $\SingX\cong \PPicY$,
\item\label{it:index2} $[\SingY:\SingX] = 2$.
\end{enumerate}
\end{prop}

\begin{proof}
The equivalence $(\ref{it:even}\iff\ref{it:exists_la})$ was proved in \cite[Proposition~3.2]{CKR18} and the equivalence $(\ref{it:even}\iff\ref{it:EE_X=Dl_Y})$ was proved in \cite[Proposition~3.4]{CKR18}. Further, \cite[Proposition~2.7]{CKR18} implies that the point~$\gp$ is even if and only if
\[
\PPicY\cong \sfrac{\PicZeroX}{2\PicZeroX} \oplus \sfrac{\ZZ}{2\ZZ}.
\]
Thus, it follows from Eq.~\eqref{eq:PicX_decomposition} that~$\gp$ is even of and only if
\[
\PPicY\cong \PPicX.
\]
Theorem~\ref{thm:EE_X=PPicX} asserts now that the right hand side is isomorphic to~$\SingX$ and this proves the equivalence $(\ref{it:even}\iff\ref{it:EE_X=PPicY})$. Finally, by \cite[Proposition~2.3.(1)]{CKR18} we have
\[
\SingY \cong \PPicY \oplus \sfrac{\ZZ}{2\ZZ}.
\]
Hence the equivalence $(\ref{it:even}\iff\ref{it:index2})$ follows from the previous part.
\end{proof}

The next lemma can be viewed as a certain variant of Dirichlet density theorem.

\begin{lem}\label{lem:density_of_cosset}
Let $k := \dim_{\FF_2}\PPicX$. For a coset $\class{\SD}+2\PicX\in \PPicX$ and a square class $\la\in \sqgd{K}$, not in~$\SingX$, denote:
\begin{gather*}
A   := \bigl\{ \gp\in X\st \class{\gp}\equiv \class{\SD}\pmod{2\PicX}\bigr\}, \\
A_+ := \bigl\{ \gp\in A\st \la\in\squares{K_\gp}\bigr\}, \qquad
A_- := \bigl\{ \gp\in A\st \la\notin\squares{K_\gp}\bigr\}.
\end{gather*}
Then:
\begin{enumerate}
\item the density of the set~$A$ equals $\delta A= \sfrac{1}{2^k}$,
\item the densities of~$A_+$ and~$A_-$ equal $\delta A_+ = \delta A_- = \sfrac{1}{2^{k+1}}$.
\end{enumerate}
\end{lem}

\begin{proof}
Pick a subset $\GB := \{\gb_1, \dotsc, \gb_k\}$ of~$X$ such that classes of $\gb_1, \dotsc, \gb_k$ form a basis of $\PPicX$ and let $\SB := \{\beta_1, \dotsc, \beta_k\}$ be a compatible basis of~$\SingX$. There are uniquely determined elements $\varepsilon_1, \dotsc, \varepsilon_k\in \FF_2$ such that
\[
\class{\SD} \equiv \sum_{i\leq k} \varepsilon_i\class{\gb_i}\pmod{2\PicX}.
\]
It follows from Proposition~\ref{prop:Pic_class_by_Legendre} that $\class{\gp}$ is congruent to $\class{\SD}$ modulo $2\PicX$ if and only if
\[
\Legendre{\beta_i}{\gp} = (-1)^{\varepsilon_i}
\]
for every $i\leq k$.  The set $\SB\cup \{\la\}$ is linearly independent since $\la\not\in \SingX$. Thus we can write:
\[
A_\pm = \Bigl\{ \gp\in X\st \Legendre{\la}{\gp} = \pm1\quad\text{and}\quad \Legendre{\beta_i}{\gp} = (-1)^{\varepsilon_i}\quad\text{for every }i\leq k\Bigr\}.
\]
Consequently Theorem~\ref{thm:Hecke_Satz_169} asserts that $\delta A_\pm = \sfrac{1}{2^{k+1}}$. The set~$A$ is a disjoint union of~$A_+$ and~$A_-$, hence $\delta A = 2\cdot \sfrac{1}{2^{k+1}} = \sfrac{1}{2^k}$.
\end{proof}

As technical as the previous lemma may appear, it has two important consequences of a more general nature. First of all, it implies that even points exist in every global function field. In fact there are always infinitely many of them. The following result substantially strengthens \cite[Proposition~3.11]{CKR18}, which was proved using different, more elementary methods.

\begin{thm}\label{thm:density_of_EP}
The set of even points has a positive density.
\end{thm}

\begin{proof}
Apply the previous lemma to $\SD = 0$.
\end{proof}

Secondly, we may now prove the promised analog of GST. Here the group of squares is replaced by the group~$\SingX$ and instead of all the points of~$K$ we just take all even points of~$K$.

\begin{thm}\label{thm:GST}
For every $\la\in \sqgd{K}$, the following conditions are equivalent:
\begin{enumerate}
\item\label{it:EE_X} $\la\in \SingX$,
\item\label{it:square_everywhere} $\la$ is a local square at every even point,
\item\label{it:square_almost_everywhere} $\la$ is a local square at almost every even point.
\end{enumerate}
\end{thm}

\begin{proof}
Assume that $\la\in \SingX$ and let $\gp\in X$ be an even point. By Proposition~\ref{prop:even_criteria} we have $\SingX = \Dl{X\setminus \{\gp\}} = \SingX\cap \squares{K_\gp}$ and so $\la\in \squares{K_\gp}$. This proves implication $(\ref{it:EE_X}\Longrightarrow\ref{it:square_everywhere})$. The implication $(\ref{it:square_everywhere}\Longrightarrow\ref{it:square_almost_everywhere})$ is trivial and the implication $(\ref{it:square_almost_everywhere}\Longrightarrow\ref{it:EE_X})$ follows from Lemma~\ref{lem:density_of_cosset}.
\end{proof}

\section{The graph of even points}\label{sec:graph}
We are now ready to generalize the relation~$\smile$, discussed in the introduction, to an arbitrary global function field. Take two even points $\gp, \gq\in X$. We say that $\gp, \gq$ are related if $\SingXp\subset \squares{K_\gq}$. We denote it then $\gp\smile \gq$. This relation was discovered in \cite{CKR18}, where it was used to study wild sets of self-equivalences of~$K$.

If $\gp\in X$ is an even point, then $[\SingXp: \SingX] = 2$ by Proposition~\ref{prop:even_criteria}. Consequently there is a (non-unique) square class $\lap\in \sqgd{K}$ such that $\SingXp$ is the disjoint union of~$\SingX$ and its coset $\lap\cdot \SingX$. In particular, $\gp$ is the unique point at which $\lap$ has an odd valuation (see condition~\ref{it:exists_la} in Proposition~\ref{prop:even_criteria}).

\begin{prop}\label{prop:smile_criteria}
Let $\gp, \gq\in X$ be two even points. The following conditions are equivalent:
\begin{enumerate}
\item\label{it:p_smile_q} $\gp\smile \gq$,
\item\label{it:lap_is_square} $\lap\in \squares{K_\gq}$,
\item\label{it:gq_splits} $\gq$ splits in $K(\sqrt{\lap})$.
\end{enumerate}
\end{prop}

\begin{proof}
The implication $(\ref{it:p_smile_q}\Longrightarrow\ref{it:lap_is_square})$ is trivial since $\lap\in \SingXp$. Next, assume that $\lap$ is a local square at~$\gq$. Select a representative of the square class~$\lap$ which has valuation~$0$ at~$\gq$. Abusing the notation slightly, denote it~$\lap$ again. The condition $\lap\in \squares{K_\gq}$ implies that $\lap(\gq)\in \squares{K(\gq)}$. Take a polynomial $f := t^2 - \lap\in K[t]$ and let $\overline{f} = t^2 - \lap(\gq)\in K(\gq)[t]$ be the reduction of~$f$ modulo~$\gq$. It is well known that $\gq$ splits in $K(\sqrt{\lap}) \cong \sfrac{K[t]}{\ideal{f}}$ if and only if $\overline{f}$ factors into linear terms, if and only if $\lap(\gq)$ is a square in the residue field $K(\gq)$. This proves $(\ref{it:lap_is_square}\Longrightarrow\ref{it:gq_splits})$. Finally assume that $\gq$ splits in $K(\sqrt{\lap})$. Then $\lap$ is a local square at~$\gq$. On the other hand, both~$\gp$ and~$\gq$ are even hence, by Proposition~\ref{prop:even_criteria}, we have 
\[
\SingXp = \SingX\dcup \lap\SingX = \Dl{X\setminus  \{\gq\}}\dcup \lap\Dl{X\setminus \{\gq\}} \subset \squares{K_\gq}.
\]
This shows the last implication and concludes the proof.
\end{proof}

\begin{rem}
Let us return for a moment to our introductory example. If $K$ is the field $\FF_q(t)$ of rational functions and $\gp, \gq$ are two prime ideals of $\FF_q[t]$ generated by some (irreducible) polynomials $f, g$ of even degrees, then we take $\lap = f$ and $\laq = g$. It is clear that $\gp\smile \gq$ if and only if $f$ is a local square at~$\gq$ if and only if $\Legendre{f}{g} = 1$. This shows that the above definition agrees with the one presented in the introduction.
\end{rem}

The following fact was proved already in~\cite{CKR18}. We repeat it here for the sake of completeness.

\begin{prop}[{\cite[Lemma~4.3]{CKR18}}]\label{prop:symmetry}
The relation~$\smile$ is symmetric.
\end{prop}

\begin{proof}
Let $\lap$ and $\laq$ be squares classes corresponding to~$\gp$ and~$\gq$, respectively. Hilbert reciprocity law asserts that
\[
\prod_{\gr\in X} (\lap, \laq)_\gr = 1.
\]
If $\gr$ is neither~$\gp$ nor~$\gq$, then both $\lap$ and~$\laq$ have even valuations at~$\gr$ and so the Hilbert symbol $(\lap, \laq)_\gr$ vanishes. Therefore the above formula simplifies to
\[
(\lap,\laq)_\gp\cdot (\lap, \laq)_\gq = 1.
\]
If we assume in addition that $\gp\smile\gq$, then $\lap\in \squares{K_\gq}$ and so $(\lap,\laq)_\gq = 1$, which implies that also $(\lap, \laq)_\gp = 1$. Now, $\ord_\gp\lap \odd$, hence $\laq$ must be a local square at~$\gp$. It follows from Proposition~\ref{prop:even_criteria} that $\gq\smile \gp$, proving a symmetry of the relation.
\end{proof}

We may now define an (infinite undirected) graph $\GE = (V,E)$, whose vertices are the even points of$~X$ and edges are defined by the relation~$\smile$, that is:
\[
V = \bigl\{ \gp\in X\st \class{\gp}\in 2\PicX\bigr\},\qquad
E = \bigl\{ (\gp,\gq)\in V\times V\st \gp\smile \gq\bigr\}.
\]

\begin{prop}\label{prop:GE_not_complete}
No vertex of~$\GE$ is adjacent to all other vertices. In particular, $\GE$~is not complete.
\end{prop}

\begin{proof}
Take an even point $\gp\in X$ and let $\mu\in \units{K}$ be an element such that $\ord_\gp\mu = 0$ and $\mu\notin\squares{K_\gp}$. Using \cite[Lemma~2.1]{LW92} we show that there is a point $\gq\in X$ and an element $\la\in \units{K}$ such that
\[
\ord_\gp(\la - \mu) \geq 1,
\qquad
\ord_\gq \la = 1
\qquad\text{and}\qquad
\la\squares{K}\in \Sing{X\setminus \{\gq\}}.
\]
Thus we can write $\la = \mu + \rho$ for some $\rho\in K_\gp$, $\ord_\gp\rho \geq 1$. It follows that in the residue field we have 
\[
(\la\mu)(\gp) = \mu(\gp)^2\in \squares{K(\gp)}.
\]
The well know correspondence between square class groups of a local field and its residue field ensures that $\la\equiv \mu\pmod{\squares{K_\gp}}$. In particular~$\la$ is not a local square at~$\gp$.

Now $\gq$ is the only point where $\la$ has an odd valuation, hence $\class{\gq}\in 2\PicX$ by Proposition~\ref{prop:even_criteria}. Moreover $\la\squares{K} = \laq$ and the classes of~$\la$ and~$\mu$ coincide in $\sqgd{K_\gp}$. Hence $\laq\notin \squares{K_\gp}$, which means that $\gq\not\smile\gp$.
\end{proof}

\begin{thm}\label{thm:diamG=2}
The graph~$\GE$ is connected and has a diameter~$2$.
\end{thm}

\begin{proof}
The diameter of~$\GE$ is greater than~$1$ by Proposition~\ref{prop:GE_not_complete}. We show that it does not exceed~$2$. Let $\gp, \gq\in X$, $\gp\neq \gq$ be two even points. Suppose that $\gp\not\smile \gq$, i.e. they are not connected by an edge of~$\GE$. We will show that there is an even point $\gr\in X$ adjacent to both of them simultaneously. To this end fix again a set $\GB\subset X$ such that the classes of its elements form a basis of $\PPicX$. Let $\SB = \{\beta_1, \dotsc, \beta_k\}$ be a compatible basis of~$\SingX$. Proposition~\ref{prop:Pic_class_by_Legendre} says that a point~$\gr$ is even if and only if
\[
\Legendre{\beta_1}{\gr} = \dotsb = \Legendre{\beta_k}{\gr} = 1.
\]
Moreover $\gr$ is adjacent to~$\gp$ and~$\gq$ if an only if
\[
\Legendre{\lap}{\gr} = \Legendre{\laq}{\gr} = 1.
\]
Theorem~\ref{thm:Hecke_Satz_169} asserts that the set
\[
\Bigl\{ \gr\in X\st \Legendre{\beta_1}{\gr} = \dotsb = \Legendre{\beta_k}{\gr} = \Legendre{\lap}{\gr} = \Legendre{\laq}{\gr} = 1\Bigr\}
\]
is not empty (in fact is infinite). Hence there is $\gr\in V$ such that $\gr\smile \gp$ and $\gr\smile \gq$ simultaneously.
\end{proof}


\begin{thebibliography}{10}

\bibitem{CKR19}
Alfred Czoga{\l}a, Przemys{\l}aw Koprowski, and Beata Rothkegel.
\newblock Wild sets in global function fields.
\newblock \url{https://arxiv.org/abs/1811.09820}.

\bibitem{CKR18}
Alfred Czoga{\l}a, Przemys{\l}aw Koprowski, and Beata Rothkegel.
\newblock Wild and even points in global function fields.
\newblock {\em Colloq. Math.}, 154(2):275--294, 2018.

\bibitem{Farkas2010}
Gavril Farkas.
\newblock The birational type of the moduli space of even spin curves.
\newblock {\em Adv. Math.}, 223(2):433--443, 2010.

\bibitem{GJRW1996}
Robert Guralnick, David~B. Jaffe, Wayne Raskind, and Roger Wiegand.
\newblock On the {P}icard group: torsion and the kernel induced by a faithfully
  flat map.
\newblock {\em J. Algebra}, 183(2):420--455, 1996.

\bibitem{LW92}
David~B. Leep and A.~R. Wadsworth.
\newblock The {H}asse norm theorem mod squares.
\newblock {\em J. Number Theory}, 42(3):337--348, 1992.

\bibitem{Putman2012}
Andrew Putman.
\newblock The {P}icard group of the moduli space of curves with level
  structures.
\newblock {\em Duke Math. J.}, 161(4):623--674, 2012.

\bibitem{Rosen02}
Michael Rosen.
\newblock {\em Number theory in function fields}, volume 210 of {\em Graduate
  Texts in Mathematics}.
\newblock Springer-Verlag, New York, 2002.

\bibitem{Sharif2013}
Shahed Sharif.
\newblock A descent map for curves with totally degenerate semi-stable
  reduction.
\newblock {\em J. Th\'{e}or. Nombres Bordeaux}, 25(1):211--244, 2013.

\bibitem{Sti93}
Henning Stichtenoth.
\newblock {\em Algebraic function fields and codes}.
\newblock Universitext. Springer-Verlag, Berlin, 1993.

\bibitem{Weil95}
Andr\'{e} Weil.
\newblock {\em Basic number theory}.
\newblock Classics in Mathematics. Springer-Verlag, Berlin, 1995.
\newblock Reprint of the second (1973) edition.

\end{thebibliography}

\end{document}